\documentclass[11pt]{amsart}

\usepackage{amsfonts,amsmath,amssymb,amsthm,amscd,amsxtra}
\usepackage{enumerate,epsfig,verbatim}

\usepackage{amsfonts,amsmath,amssymb,amsthm,amscd,amsxtra}
\usepackage{enumerate,verbatim}

\usepackage[all,2cell,ps]{xy}
\usepackage[notcite,notref]{}
\usepackage[pagebackref]{hyperref}

\usepackage[all,2cell,ps]{xy}
\usepackage{amsfonts}
\usepackage[margin=1.4in]{geometry}
\usepackage{color}
\usepackage[notcite,notref]{}
\usepackage{url}
\usepackage{datetime}
\usepackage{mathtools}

\usepackage[all,2cell,ps]{xy}
\usepackage{amssymb}
\usepackage{amsmath}
\usepackage{amsfonts}
\usepackage{amsmath}
\usepackage{amsthm}
\usepackage{amssymb}
\usepackage{amscd}
\usepackage{amsfonts}
\usepackage{amsxtra}     
\usepackage{epsfig}
\usepackage{verbatim}
\usepackage[all]{xypic}
\SelectTips{cm}{}

\def\latex/{{\protect\LaTeX}}
\def\latexe/{{\protect\LaTeXe}}
\def\amslatex/{{\protect\AmS-\protect\LaTeX}}
\def\tex/{{\protect\TeX}}
\def\amstex/{{\protect\AmS-\protect\TeX}}
\def\bibtex/{{Bib\protect\TeX}}
\def\makeindx/{\textit{MakeIndex}}

\usepackage{amsmath}
\usepackage{amsthm}
\usepackage{amssymb}
\usepackage{amscd}
\usepackage{amsxtra}     
\usepackage{epsfig}
\usepackage{verbatim}
\usepackage[all]{xypic}
\usepackage{enumerate}

\theoremstyle{plain} 

\newtheorem{thm}{Theorem}[section]

\newtheorem{cor}[thm]{Corollary}

\theoremstyle{definition}
\newtheorem{chunk}[thm]{\hspace*{-1.065ex}\bf}

\newtheorem{ques}[thm]{Question}

\newtheorem{rmk}[thm]{Remark}

\newcommand{\bZ}{\mathbb{Z}}

\newcommand{\fm}{\mathfrak{m}}

\newcommand{\bR}{\mathbb{R}}
\newcommand{\bN}{\mathbb{N}}
\newcommand{\ltensor}{\tensor^{\bf{L}}}

\newcommand{\qpd}{\operatorname{qpd}}
\newcommand{\susp}{\mathsf{\Sigma}}

\newcommand{\ZZ}{\mathbb{Z}}

\newcommand{\tensor}{\otimes}

\DeclareMathOperator{\ds}{\displaystyle}

 \DeclareMathOperator{\Tor}{Tor}

\DeclareMathOperator{\Hom}{Hom}
\DeclareMathOperator{\Ho}{H}
\DeclareMathOperator{\CI}{\textnormal{CI-dim}}

 \DeclareMathOperator{\Supp}{Supp}
 \DeclareMathOperator{\Spec}{Spec}

 \DeclareMathOperator{\pd}{pd}
 
 \DeclareMathOperator{\cx}{cx}

 \DeclareMathOperator{\edim}{embdim}
 \DeclareMathOperator{\depth}{depth}

 \DeclareMathOperator{\htt}{height}

\newcommand{\Ann}{\textup{Ann}}

\newcommand{\zdr}[2]{{\boldsymbol\top}_{\hskip-2pt #1}{#2}}

\def\urltilda{\kern -.15em\lower .7ex\hbox{\~{}}\kern
  .04em}\def\urldot{\kern -.10em.\kern -.10em}\def\urlhttp{http\kern
  -.10em\lower -.1ex\hbox{:}\kern -.12em\lower 0ex\hbox{/}\kern
  -.18em\lower 0ex\hbox{/}} 

\newcommand{\bb}{\left[ \begin{smallmatrix}}
\newcommand{\eb}{\end{smallmatrix} \right]}

\begin{document}

\title[Syzygies and tensor product of modules]{syzygies and tensor
  product of modules }

\author{Olgur Celikbas} \address{Department of Mathematics, 323
  Mathematical Sciences Bldg, University of Missouri,
  Columbia, MO 65211 USA} \email{celikbaso@missouri.edu}

\author{Greg Piepmeyer} \address{Department of Statistics, 146
  Middlebush Hall, University of Missouri, Columbia, MO
  65211 USA} \email{ggpz95@mail.missouri.edu}

\thanks{2010 {\em Mathematics Subject Classification.} 13C12, 13C13}
\keywords{Derived depth formula, complex, complexity, complete intersection dimension, Serre's condition, tensor product, torsion, Tor-rigidity, New
  Intersection Theorem} 
\maketitle{}

\begin{abstract} 
We give an application of the New Intersection Theorem and prove the following: let $R$ be a local complete intersection ring of codimension $c$ and let $M$ and $N$ be nonzero finitely generated $R$-modules. Assume $n$ is a nonnegative integer and that the tensor product $M\tensor_{R}N$ is an $(n+c)$th syzygy of some finitely generated $R$-module. If $\Tor^{R}_{>0}(M,N)=0$, then both $M$ and $N$ are $n$th syzygies of some finitely generated $R$-modules.
\end{abstract}

\section{Introduction}

Auslander's rigidity theorem \cite[2.2]{Au} states that
finitely generated modules over unramified (or equi-characteristic)
regular local rings are {\em Tor-rigid}, that is, if $R$ is
such a ring, and $M$ and $N$ are 
such modules, then the vanishing of $\Tor^{R}_{n}(M,N)$ for some
nonnegative integer $n$ forces the vanishing of all subsequent Tor
modules. This result yielded several consequences and received considerable attention; Lichtenbaum \cite{Li} extended it to all regular rings by proving the ramified case.

Auslander \cite{Au} made use of his rigidity result, analyzing \emph{torsion} submodules of the
tensor products of finitely generated modules over unramified regular local
rings. Recall that, if $R$ is a commutative Noetherian ring, then the \emph{torsion} submodule $\zdr R M$ of an $R$-module $M$ is defined as the set $\{x\in M: rx=0 \text{ for
  some non zero-divisor } r\in{R} \}$. $M$ is called
\emph{torsion-free} if $\zdr R M=0$, and \emph{torsion} if $\zdr R M=M$; see
(\ref{Snfact}).
An important consequence of Auslander's rigidity theorem we are concerned with in this paper is:

\begin{thm} (Auslander \cite[3.1]{Au}) \label{rigid} %
  Let $R$ be an unramified  (or equi-characteristic) regular local ring and let $M$ and $N$ be finitely generated $R$-modules. Assume that the tensor product
  $M\otimes_{R}N$ is torsion-free. Then (i) $\Tor^{R}_{>0}(M,N)=0$,
  and (ii) $M$ and $N$ are both torsion-free.
\end{thm}

Let us remark that regular local rings are complete intersections of
codimension zero.  Furthermore, in this setting, torsion-freeness is
equivalent to \emph{Serre's condition} $(S_{1})$; see
(\ref{Sn}). Therefore it is quite natural to ask whether a similar
conclusion of Theorem \ref{rigid} holds, for tensor products of
modules satisfying higher $(S_{n})$ conditions, over complete
intersection rings of higher codimension. Huneke and R.~Wiegand
\cite{HW1} pioneered this question and extended many of Auslander's
results to \emph{hypersurfaces}, that is, complete intersections of
codimension one. They proved, with an unavoidable hypothesis, that if
$R$ is a hypersurface ring and the tensor product $M\otimes_{R}N$ is
reflexive, then $\Tor^{R}_{>0}(M,N)=0$. Their result is known as the
\emph{Second Rigidity Theorem}:

\begin{thm} (Second Rigidity Theorem, \cite[Theorem 1]{HW3}) \label{secondrigidity} 
Let $R$ be a local hypersurface and let $M$ and $N$ be finitely generated $R$-modules. Assume $M$ or $N$ has constant rank. Assume further that the tensor product $M\otimes_{R}N$ is reflexive. Then $\Tor^{R}_{>0}(M,N)=0$.
\end{thm}  

Theorem \ref{secondrigidity} is a substantial generalization of
Auslander's result, stated as Theorem \ref{rigid}, to hypersufrace
rings; over such rings, Serre's condition $(S_{2})$ is equivalent to
reflexivity, or being a second syzygy module; see (\ref{Snfact}).
Recall that $M$ is said to have {\em constant rank} if there exists an
integer $r$ such that $M_{p}\cong R_{p}^{(r)}$ for all associated
primes $p$ of $R$. For example, finitely generated modules of finite projective dimension have constant rank \cite[1.4.6]{BH}. The constant rank hypothesis in Theorem
\ref{secondrigidity} is necessary; see \cite[1.2(1)]{HW1}.

A nice consequence of Theorem \ref{secondrigidity} is that one of the modules considered must be reflexive. This fact is not explicitly stated in \cite{HW1} so we discuss this next; see the proof of \cite[2.7]{HW1} for more details. Note that $\depth(0)=\infty$; this is indeed the correct convention which is needed for several fundamental results in commutative algebra such as the \emph{depth lemma} \cite[1.2.9]{BH}; see also \cite{HW3}.

\begin{rmk} \label{dis} Let $R$, $M$ and $N$ be as in Theorem \ref{secondrigidity} with $M\neq 0 \neq N$. Assume that M is the module of constant rank, say of rank $r$. Then $r\neq 0$; otherwise $\zdr R M = M$ and this would imply that $\zdr R (M\otimes_{R}N) = M\otimes_{R}N$, which is not possible as $M\otimes_{R}N$ is torsion-free. It follows, as each prime ideal contains a minimal prime, that $\Supp(M)=\Spec(R)$. Furthermore the vanishing of $\Tor^{R}_{>0}(M,N)$ yields an equality of depths, known as the depth formula,
$\depth(M) + \depth(N) = \depth(R) + \depth(M \otimes_R N)$; see (\ref{DF}). Now, if $p\in \Supp(N)$, then one can localize the depth formula at $p$ and obtain an equality all of whose  terms are finite. As $\depth(R_{p})-\depth_{R_{p}}(M_{p})\geq 0$, we have:
$$\depth_{R_{p}}(N_{p})\geq \depth_{R_{p}}(N_{p} \otimes_{R_{p}} M_{p}) \geq \text{min} \{2, \htt(p) \}.$$ Therefore $N$ satisfies $(S_{2})$, that is, $N$ is reflexive; see (\ref{Sn}) and (\ref{Snfact}).
\end{rmk}

Our main aim in this paper is to generalize the observation we discussed in \ref{dis} to complete intersections of arbitrary codimension for modules that may not have constant rank. We obtain such a generalization in Theorem \ref{Cx}; a special case of our argument can be stated as follows:

\begin{thm} \label{resultintro} 
Let $R$ be a local complete intersection ring of codimension $c$ and let $M$ and $N$ be nonzero finitely generated $R$-modules. Assume that $M\otimes_{R}N$ satisfies $(S_{n+c})$ for some nonnegative integer $n$, equivalently $M\otimes_{R}N$ is an $(n+c)$th syzgy of some finitely generated $R$-module. If $\Tor^{R}_{>0}(M,N)=0$, then both $M$ and $N$ satisfy $(S_{n})$.
\end{thm}

The main ingredient that makes the argument in Remark \ref{dis} work is the vanishing of $\Tor^{R}_{>0}(M,N)$, which follows from the fact that $M$ or $N$ has constant rank over a hypersurface.  We should note that we do not know whether the reflexivity of the tensor product of two finitely generated modules, either of which has constant rank, implies the vanishing of $\Tor^{R}_{>0}(M,N)$ over complete intersections that are not hypersurfaces.  Therefore, in order to generalize the observation of Remark \ref{dis} to complete intersection rings, it is reasonable to assume the vanishing of such Tor modules, as we do in Theorem \ref{Cx}. Since modules over domains have constant rank, it seems worth posing the following question; see also \cite[3.14]{Ce}.

\begin{ques} Let $R$ be a complete intersection \emph{domain} and let $M$ and $N$ be nonzero finitely generated $R$-modules. If $M\otimes_{R}N$ is reflexive, then must $\Tor^{R}_{>0}(M,N)=0$? What if $M$, $N$ and $M\otimes_{R}N$ are all maximal Cohen-Macaulay?
\end{ques}

Theorem \ref{Cx} is stated and proved for homologically bounded complexes.  We use this generality for two reasons: one of the main ingredients of our proof, a generalization of the New Intersection Theorem, was first developed for complexes \cite[3.3]{SY2}; see also (\ref{NIThm}.1).  Secondly, a reader interested merely in modules may follow our arguments word for word by simply replacing each $\inf$ and $\sup$ by zero, so there is no penalty for this extra generality.

We are also interested in analyzing the torsion of \emph{both} of the modules $M$ and $N$ considered in the Second Rigidity Theorem; see Theorem \ref{secondrigidity}. We observed in \ref{dis} that the support of $M$ or $N$ is the whole spectrum of the ring so that one can use the depth formula and easily deduce the reflexivity of either of those modules; see (\ref{DF}). One can also see, following the arguments given in the proof of the Second Rigidity Theorem, that  \emph{both} of the modules considered do satisy $(S_{1})$, that is, both are torsion-free modules; see the discussion following Corollary \ref{cor3}.  Indeed our main argument shows that this fact carries over to complete intersection rings of arbitrary codimension; see Corollary \ref{corfinal}. 

We should note that, if $M$ and $N$ are arbitrary nonzero finitely generated modules over a complete intersection ring $R$ that is not a hypersurface, then the vanishing of $\Tor^{R}_{\gg 0}(M,N)$ does not necessarily imply that $M$ or $N$ has finite projective dimension; for an example see \cite[4.2]{Jor}. Therefore, when $M\otimes_{R}N$ is reflexive and $\Tor^{R}_{>0}(M,N)=0$, the support of both $M$ and $N$ may be a proper subset of $\Spec(R)$, and hence one cannot repeat the arguments of \ref{dis} to examine the reflexivity of $M$ or $N$: the depth formula does not immediately yield any information because localizing the depth formula at a prime ideal that is not in the support of \emph{both} $M$ and $N$ yields the void equality $\infty=\infty$; see also \cite{HW3}. We are able to overcome this difficulty by using the New Intersection Theorem in the proof of Theorem \ref{Cx}; see also Corollaries \ref{cor2} and \ref{corintro}.

One can also ask whether there exists an analogous observation to that of \ref{dis} for the torsion-free tensor products of modules:

\begin{ques} \label{impossible} Let $R$ be a hypersurface \emph{domain} and let $M$ and $N$ be nonzero finitely generated $R$-modules. If $M\otimes_{R}N$ is torsion-free, then must $M$ or $N$ be torsion-free? 
\end{ques}

One can find several interesting results related to Question \ref{impossible} in \cite{CIPW} where it was studied in detail via different techniques. 

\section{Definitions and Preliminary Results}

Our proofs use computations
in the derived category over commutative Noetherian rings.  Again, our 
arguments may be read just for modules by simply replacing each
$\inf$ and each $\sup$ by zero.  However, for readers interested in 
more details for the hyperhomology techniques, we
suggest \cite[Appendix]{Larsbook} or \cite{Larsbook2}.  

A local ring $(R,\fm,k)$ is a \emph{complete intersection} if the
defining ideal of some (equivalently every) Cohen presentation of the
$\fm$-adic completion $\widehat{R}$ of $R$ can be generated by a
regular sequence. If $R$ is such a ring, then $\widehat{R}$ has the
form $Q/(\underline{x})$, where $\{\underline{x}\}$ is a $Q$-regular
sequence and $Q$ is a ring of formal power series over the residue
field $k$, or over a complete discrete valuation ring with residue
field $k$.  The \emph{codimension} of $R$ is the nonnegative integer
$\dim_{k}(\fm/\fm^2)-\dim(R)$. A complete intersection of codimension
one is called a \emph{hypersurface}.  We should note that Huneke and
R.~Wiegand define a hypersurface ring in \cite{HW1} as a quotient of
a regular local ring by a nonzero element. Although this differs from
our definition, 
the difference does not affect our results; constant rank and the
$(S_{n})$ conditions are preserved under completion; see the
discussion following Theorem \ref{secondrigidity} and (\ref{Snfact}).

Let $R$ be a commutative Noetherian ring. We grade complexes
$(M,\partial^{M})$ of $R$-modules homologically.  The $n$th homology
module of $M$ is denoted by $\Ho_n(M)$. The \emph{supremum} and
\emph{infimum} of $M$ are given by
 \begin{equation*}
 \sup(M)=\sup\{i\in \ZZ: \Ho_{i}(M)\neq 0\} \text{ and }
 \inf(M)=\inf\{i\in \ZZ: \Ho_{i}(M)\neq 0\}. 
 \end{equation*}
 $M$ is \emph{homologically bounded} provided that both $\sup(M)$ and
 $\inf(M)$ are finite.
 
\begin{chunk} \textbf{Conventions.} \label{why} Through the paper $R$
  denotes a local ring, that is, a commutative Noetherian ring with
  unique maximal ideal $\fm$ and residue class field $k=R/\fm$. All
  modules considered over $R$ are finitely generated. By an $R$-complex, we mean a
  complex of finitely generated $R$-modules which is always assumed to
  be homologically bounded.
\end{chunk}

\begin{chunk} \textbf{Support and Dimension.} \label{dim} %
  The \emph{annihilator} and \emph{support} of an $R$-complex $M$ are
  defined to be the following:
\begin{align*} \notag 
\Ann(M) =\bigcap_{i\in \ZZ} \Ann(\Ho_{i}(M))  \text{ and }
\Supp(M)=\bigcup_{i\in \ZZ}\Supp(\Ho_{i}(M)).
\end{align*}
The (Krull) dimension of an $R$-complex $M$ is defined as
\begin{align*} \notag 
\dim(M) &= \sup \{ \dim(R/p) - \inf(M_{p}): p\in \Supp(M)  \}  =\sup\{\dim_{R}\Ho_{i}(M)-i: i \in \ZZ\}.
\end{align*}
These definitions extend the usual concepts for modules, see
\cite[A.8.4]{Larsbook}.
\end{chunk}

\begin{chunk} \textbf{Depth.} \label{depth} %
The \emph{depth} of an $R$-complex $M$ is defined by
\begin{equation} \tag{\ref{depth}.1}
\depth(M) = n - \sup(K \ltensor_R M).
\end{equation} 
Here $K$ is the Koszul complex on a set of generators $x_1, \ldots,
x_n$ of $\fm$; see \cite[\S 2]{I} and \cite[(A.6.1)]{Larsbook}.  It is
known that this definition is independent of the choice of generators
of $\fm$ \cite[1.3]{I}. Notice, by our convention for $R$-complexes
(\ref{why}), $\depth(0)=\infty$ and $\depth(M)$ is finite if $\Ho(M)
\neq 0$.
\end{chunk}

The Koszul complex $K$ is a finite free complex; hence $K \tensor_R M
\simeq K \ltensor_R M$. For any fixed choice of generators of
$\mathfrak{m}$, the quantity $n$ and the collection $\{ H_i(K
\tensor_R M) \colon i \in \bZ \}$ of homology modules of $K \tensor_R
M$ are invariant under suspension of $M$, but the depth is not: the suspension $\susp M$ of $M$ satisfies $\depth(\susp M) = \depth(M) - 1$.  

\begin{chunk} \textbf{Serre's Condition.}\label{Sn} %
  We say an $R$-complex $M$ satisfies \emph{Serre's condition}
  $(S_{n})$ provided that the following inequality holds for all $p
  \in \Supp(M)$:
\begin{equation*} 
\depth_{R_{p}}(M_p) + \inf(M_p) \geq \min\{n, \htt(p) \}.  
\end{equation*}

When $M$ is a finitely generated $R$-module, this definition agrees
with the classical one defined by Evans and Griffith \cite{EG}, that
is, $\depth(M_p) \geq \min \{n, \htt(p) \}$ for all $p \in
\Supp(M)$. Note that their definition for Serre's condition for
finitely generated modules \emph{differs} from the one given in
\cite[5.7.2.I]{EGA} and \cite[page 62]{BH}.
\end{chunk}

We record the following facts about Serre's condition $(S_{n})$.

\begin{chunk} \textbf{Facts.} \label{Snfact} (\cite[3.8]{EG},
  \cite[1.7]{Har2}, \cite[1.3]{LW} and \cite[Proposition 6]{Sam}) 
Assume $R$ is a Gorenstein local ring and $M$ is a finitely generated
$R$-module. 
\begin{enumerate}[(i)]
\item The following conditions are equivalent:
\begin{enumerate}
\item $M$ satisfies $(S_{n})$.
\item $M$ is an $n$th syzygy module.
\item All $R$-regular sequences of length at most $n$ are also $M$-regular.
\end{enumerate}

\item $M$ is \emph{torsion-free}, that is, the natural map $M \to
  \Hom(\Hom(M,R), R)$ is injective, if and only if $M$ satisfies
  $(S_{1})$.
\item $M$ is \emph{reflexive}, that is, the natural map $M \to
  \Hom(\Hom(M,R), R)$ is bijective, if and only if $M$ satisfies
  $(S_{2})$.
\item If $R\to S$ is a flat ring map and $M$ satisfies $(S_{n})$ over
  $R$, then $M\otimes_{R}S$ satisfies $(S_{n})$ over $S$.
\end{enumerate}
\end{chunk}

\begin{chunk} \textbf{Projective Dimension.} \label{pd} %
  Recall that the $n$th \emph{Betti number} of an $R$-complex $M$ is
  defined as $\beta_n^R(M) \coloneqq \dim_k(\Tor_{n}^R(M,
  k))=\dim_k(\Ho_{n}(M\ltensor_{R}k))$. The \emph{projective
    dimension} of $M$ can be given as follows
  \cite[(A.5.3)]{Larsbook}:
\begin{equation} \notag{}
\pd(M)=\sup\{n\in\ZZ: \beta_n^R(M)\neq 0 \}. 
\end{equation}
\end{chunk}

\begin{chunk} \textbf{Quasi-projective and Complete Intersection
    Dimension.} \label{qCI} 
  A diagram of local ring maps $R \to R' \twoheadleftarrow Q$ is
  called a {\em quasi-deformation} provided that $R \to R'$ is flat
  and the kernel of the surjection $R' \twoheadleftarrow Q$ is
  generated by a $Q$-regular sequence, see \cite[1.1]{AGP}.

  The \emph{quasi-projective dimension} of an $R$-complex $M$
  \cite[3.3]{Av1}, see also \cite[3.1]{SY2}, is:
\begin{equation*} 
  \qpd_R(M) = \inf\{ \pd_Q(M \ltensor_R R') \ | \  R \to R'  \twoheadleftarrow
  Q \ {\text{is a quasi-deformation}} \}.
\end{equation*}

The \emph{complete intersection dimension} of $M$, see \cite[1.2]{AGP}
and \cite[3.1]{Sean}, is quite similar:
\begin{equation*} 
\CI_R(M) = \inf\{ \pd_Q(M \ltensor_R R') - \pd_Q(R') \ | \ 
R \to R' \twoheadleftarrow Q \ {\text{is a quasi-deformation}}\}.
\end{equation*}

Notice $\CI(M)<\infty$ if and only if $\qpd_R(M)<\infty$. If $R$ is a
complete intersection, then $\CI(M)<\infty$
\cite[1.3]{AGP}. Furthermore, by \cite[3.3]{Sean}, we have:
\begin{equation} \tag{\ref{qCI}.1}
\sup(M)\leq\CI(M)\leq \pd(M).  
\end{equation}
\end{chunk}

\begin{chunk} \textbf{Derived Depth Formula.} \label{DF} 
  Let $M$ and $N$ be $R$-complexes. We say that the pair $(M,N)$
  satisfies the \emph{derived depth formula} provided the following
  equality holds:
\begin{equation} \notag{}
\depth(M) + \depth(N) = \depth(R) + \depth(M \ltensor_R N).  
\end{equation}
Christensen and Jorgensen proved in \cite[5.4(a)]{CJ} that if
$\Tor_{\gg 0}^{R}(M,N)=0$ and $M$ has finite complete intersection
dimension, see (\ref{qCI}), then the pair of complexes $(M,N)$
satisfies the derived depth formula.

The depth formula, for the tensor product of finitely generated modules, 
is initially due to Auslander \cite{Au}; see also Foxby \cite{Foxby} and Iyengar \cite{I}.
\end{chunk}

\begin{chunk} \textbf{Complexity.} \label{cx} %
  The \emph{complexity} of an $R$-complex $M$ is a measure for the
  polynomial growth rate of its Betti numbers and is defined as
  follows, see \cite[4.2]{Av2} and \cite[5.1]{AGP}.  
\begin{equation*}
  \cx_R(M) = \inf \{ r \in \bN \ | \ \exists \beta \in \bR {\text{ so
      that }} \beta_n^R(M) \leq \beta 
  n^{r-1} \ \forall n \gg 0 \} 
\end{equation*}

Note that $\cx(M)=0$ if and only if $\pd(M)<\infty$. Moreover, the
finiteness of complete intersection dimension implies the finiteness
of complexity; this is essentially due to Gulliksen
\cite[3.3]{Gul2}. More precisely, when $\CI_{R}(M)<\infty$, the
complexity $\cx(M)$ cannot exceed $\edim(R)-\depth(R)$
\cite[3.10(v)]{Sean}.
\end{chunk}

\begin{chunk} \label{qpdviaCI} %
  Avramov-Gasharov and Peeva \cite[5.11]{AGP} obtained a relationship
  between quasi-projective and complete intersection dimension using
  complexity: if $M$ is an $R$-complex, then the following equality
  holds:
\begin{equation*} 
\qpd_R(M)=\CI(M)+\cx_{R}(M).
\end{equation*}

This equality was originally obtained in \cite{AGP} for finitely
generated modules. Its proof mainly depends on a factorization result
of quasi-deformations \cite[5.9]{AGP} which can be restated for
complexes. A proof of this equality for $R$-complexes follows the
arguments in \cite[5.9]{AGP} using the analogous results in
\cite[\S5]{AvSun} and \cite{Sean}.
\end{chunk}

\begin{chunk} \textbf{Auslander-Buchsbaum Formula.}\label{ABF} %
  Complexes of finite projective dimension also have finite complete
  intersection dimension. The finiteness of either of these
  homological dimensions implies an equality, known as the
  \emph{Auslander-Buchsbaum formula}: if $M$ is an $R$-complex and if
  $\pd(M)<\infty$ (respectively, if $\CI(M)<\infty$), then $\pd_R(M) =
  \depth(R) - \depth(M)$ (respectively, $\CI(M) = \depth(R) -
  \depth(M)$).

  If $\qpd_R(M)<\infty$, then the Auslander-Buchsbaum formula and
  (\ref{qpdviaCI}) give that:
\begin{equation} \tag{\ref{ABF}.1}
\qpd_R(M)=\depth(R) - \depth(M)+\cx_{R}(M).
\end{equation}
\end{chunk}

\begin{chunk} \textbf{New Intersection Theorem.} \label{NIThm} %
  The \emph{New Intersection Theorem} of Hochster \cite{Hochster},
  Peskine and Szpiro \cite{PS} and Roberts \cite{R1}, \cite{R2} states
  that, if $M$ and $N$ are finitely generated $R$-modules, then:
\begin{equation} \notag{}
\dim_{R}(N)\leq \pd_{R}(M)+\dim_{R}(M\otimes_{R}N).
\end{equation}
This inequality can also be given for $R$-complexes, see for example
\cite[\S 18]{Larsbook2}.

If $M$ and $N$ are $R$-complexes with $\Ho(M) \neq 0$ and
$\qpd(M)<\infty$, then Sharif and Yassemi \cite[3.3]{SY2}, see also
(\ref{qCI}), generalized the New Intersection Theorem as follows:
\begin{equation}  \tag{\ref{NIThm}.1}
\dim_{R}(N)\leq \qpd(M)+\dim_{R}(M\otimes_{R}N).
\end{equation}
\end{chunk}

We finish this section by recording some inequalities that will be
used later.

\begin{chunk} \label{ineq} %
  Let $M$ and $N$ be $R$-complexes such that $\CI_R(M)< \infty$ and
  $(M,N)$ satisfies the derived depth formula. Then it follows from
  (\ref{ABF}.1) and (\ref{DF}) that:
\begin{align*}
\qpd_{R}(M) & = \depth(R) - \depth(M)+\cx_{R}(M)   \\
& = \depth(N) - \depth(M \ltensor_{R} N) +\cx_{R}(M).  
\end{align*}
\end{chunk}

\begin{chunk} \label{ineqLars} %
  Let $M$ and $N$ be $R$-complexes. Then, by \cite[(A.6.2)]{Larsbook},
  we have 
\begin{equation*} 
  \depth_{R}(N) \leq \depth_{R_{p}}(N_{p})+ \dim(R/p)  \text{ for all
  } p\in \Spec(R).  
\end{equation*}
\end{chunk}

\section{An application of the New Intersection Theorem}

One can use the depth formula, see (\ref{DF}), 
to make local conclusions about either
$M$ or $N$ \emph{only} on the intersection of their supports
\cite{HW3}.  Localizing the depth formula at a prime ideal that is
outside of the support of $M$ or $N$ yields a void equality;
$\depth(0)=\infty$, see (\ref{why}), (\ref{depth}) and also \cite[3.3]{HW1}.

In Theorem \ref{Cx} we make use of the depth formula (\ref{DF}) and
analyze the depth of the derived tensor product $M\ltensor_{R}N$ of
$R$-complexes. Our aim is to determine certain conditions so that, if
$M\ltensor_{R}N$ satisfies $(S_{n+r})$, where $r=\cx_{R}(M)$, then $N$
satisfies $(S_{n})$, see also (\ref{Sn}) and (\ref{cx}).  We combine
an argument, analogous to that of Yoshida, see his proof \cite[4.1, Case 2]{Yoshida},
with the inequality (\ref{NIThm}.1), due to Sharif and Yassemi
\cite{SY2}, that generalizes the New Intersection Theorem. 
The result we obtain is, to our knowledge, new, even for finitely generated modules.

\begin{thm} \label{Cx} 
  Let $R$ be a local ring, $M$ and $N$ be $R$-complexes with
  $\Ho(M)\neq 0$ and let $n$ be a nonnegative integer. Set
  $r=\cx_{R}(M)$ and assume the following conditions:
\begin{enumerate} [(i)]
\item $\CI_{R}(M)<\infty$,
\item $M\ltensor_{R}N$ satisfies $(\ds{S_{n+r}})$, and
\item $\Tor_{\gg 0}^{R}(M,N)=0$.
\end{enumerate}
Then $N$ satisfies $(S_{n})$.
\end{thm}

\begin{proof} We shall prove that, for all $p \in \Supp(N)$, the
  following inequality holds:
\begin{equation} \label{depth for N} \tag{\ref{Cx}.1} 
\depth_{R_{p}}(N_{p})+\inf(N_{p}) \geq \min\{n, \htt(p) \}.  
\end{equation}

We start by recording some observations. Note that, by (iii),
$M\ltensor_{R}N$ satisfies our convention to be an $R$-complex, see
(\ref{why}). Notice also that the complexity, Serre's condition,
finiteness of complete intersection dimension, the vanishing of
$\Tor^{R}_{\gg 0}(M,N)$ as well as the conclusion of the theorem are
invariant under suspension of $M$, see (\ref{Sn}), (\ref{qCI}) and
(\ref{cx}). Therefore we may assume that $\inf(M)=0$. Furthermore, for
all $q\in \Supp(M)$, the following inequality holds, see (\ref{qCI})
and (\ref{ABF}).
\begin{equation}  \notag 
0=\inf(M)\leq \inf(M_{q}) \leq \sup(M_{q}) \leq
\CI_{R_{q}}(M_{q})=\depth(R_{q})-\depth(M_{q}) 
\end{equation}
In particular, setting $q=\fm$, we see that $\depth(M)\leq
\depth(R)$.

The hypotheses imply that $(M,N)$ satisfies the derived depth formula
(\ref{DF}); hence
\begin{equation}  
\label{Cx DF} \notag 
\depth(M) + \depth(N) = \depth(R) + \depth(M \ltensor_{R} N).  
\end{equation}
Depth formula 
together with the fact that $\depth(M)\leq
\depth(R)$ gives
\begin{equation}  \tag{\ref{Cx}.2}
\depth(N) \geq \depth(M \ltensor_{R} N)
\end{equation}
We know, by \cite[(A.4.16)]{Larsbook}, that
$\inf(M\ltensor_{R}N)=\inf(M)+\inf(N)=\inf(N)$. Hence (ii) and
(\ref{Cx}.2) show that 
\begin{align*}  \notag 
\depth(N)+\inf(N) & \geq  \depth(M\ltensor_{R}N)+\inf(M\ltensor_{R}N)
\geq \min\{n+r, \dim(R) \}. 
\end{align*}
This yields, for any $p\in \Supp(N)$, the following inequalities:
\begin{align*} 
\depth_{R_{p}}(N_{p})+\inf(N_{p}) & \geq \depth_{R_{p}}(N_{p})+\inf(N)
\\  \tag{\ref{Cx}.3} 
& \geq \depth(N)+\inf(N) -\dim(R/p) \\ 
& \geq \min\{n+r, \dim(R) \} -\dim(R/p). 
\end{align*}
Here the second inequaliy follows from
(\ref{ineqLars}).  

Let $p\in \Supp(N)$. We now divide our proof into two main cases: 
\begin{enumerate}
\item $\dim(R)\leq n+r$ or $p \in \Supp(M)$.
\item $\dim(R)> n+r $ and $p\not \in \Supp(M)$.
\end{enumerate}

\noindent {\sc{Case 1}}. %
Assume $p\in \Supp(M)$.  Then $p\in \Supp(M\ltensor_{R}N)=\Supp(M)
\cap \Supp(N)$ and hence $\Ho(M_{p})\neq 0\neq \Ho(N_{p})$. The
hypotheses and the conclusion are preserved under localization at
$p$. Hence we may consider the pair $(M_{p}, N_{p})$ over $R_{p}$. 
Therefore (\ref{Cx}.3) yields
\begin{equation*}  
\depth(N)+\inf(N) \geq  \depth(M\ltensor_{R}N)+\inf(M\ltensor_{R}N)
\geq \min\{n+r, \dim(R) \},   
\end{equation*}
which justifies (\ref{depth for N}).  Now assume $\dim(R) \leq
n+r $.  Then it follows from (\ref{Cx}.3) that
\begin{align*} 
\depth_{R_{p}}(N_{p})+\inf(N_{p}) & \geq \min\{n+r, \dim(R) \} -\dim(R/p) \\
& \geq \dim(R)-\dim(R/p) \\
& \geq \htt(p).  
\end{align*}
Therefore we have established (\ref{depth for N}) in this case, too.  
\vspace*{1ex}

\noindent {\sc{Case 2}}. %
Assume $p \notin \Supp(M)$ and that $\dim(R)> n+r $. Recall that $p\in
\Supp(N)$. Let $q$ be a minimal prime ideal of $p + \Ann_R(M)$. Then
$q\in \Supp(M\ltensor_{R}N)$; see (\ref{dim}).

Suppose first that $\htt(q)\leq n+r$. Then, since $\Ho(M_{q})\neq
0\neq \Ho(N_{q})$, we can consider the pair $(M_{q}, N_{q})$ over
$R_{q}$ and deduce from {\sc{Case 1}} that (\ref{depth for N})
holds. Therefore we proceed under the assumption that
$ \htt(q)> n+r$.  

The New Intersection Theorem (\ref{NIThm}.1), applied to the pair
$(M_{q}, (R/p)_{q})$ over the ring $R_{q}$, gives 
\begin{equation} \label{NIT} \notag 
  \qpd_{R_q} (M_q) + \dim_{R_{q}}(M_q \ltensor_{R_q} R_q/pR_{q}) \geq
  \dim_{R_q}( R_q/pR_{q}).  
\end{equation}
Furthermore the inequality (\ref{ineqLars}) applied to the $R_{q}$
complex $N_{q}$ yields that
\begin{equation} \notag 
  \depth_{R_p}(N_p) \geq \depth_{R_q}(N_q) - \dim_{R_q}(R_q/p_q).
\end{equation}
These last two observations imply
\begin{equation} \label{NIT2} \tag{\ref{Cx}.4} %
  \depth_{R_p}(N_p) \geq \depth_{R_q}(N_q) - \big( \qpd_{R_q} (M_q) +
  \dim_{R_q}(M_q \ltensor_{R_q} R_q/pR_q) \big).
\end{equation}
Note that $\Supp_{R_{q}}(M_q\ltensor_{R_q}
R_q/pR_q)=\{qR_{q}\}$.  Hence, by (\ref{dim}), we deduce
\begin{align*} \label{NIT2} \notag 
  \dim(M_q\ltensor_{R_q} R_q/pR_q) 
  & =\dim(R_{q}/qR_{q})-\inf\big((M_q\ltensor_{R_q} R_q/pR_q)_{qR_{q}}
  \big)
  =-\inf(M_{q}).  
\end{align*}
Therefore, substituting this into (\ref{Cx}.4), we have
\begin{equation} \label{NIT3} \tag{\ref{Cx}.5}
 \depth_{R_p}(N_p)  \geq   \depth_{R_q}(N_q) - \qpd_{R_q} (M_q)  +
 \inf(M_{q}).  
\end{equation}
Note that (\ref{ineq}) gives 
\begin{equation} \label{NIT4} \tag{\ref{Cx}.6} 
\qpd_{R_q} (M_q) = \depth_{R_q}(N_q) - \depth_{R_q}(M_q\ltensor_{R_q}
N_q) + \cx(M_q).  
\end{equation} 
Combining (\ref{NIT3}) with (\ref{NIT4}), we obtain
\begin{equation*}
  \depth_{R_p}(N_p) \geq   \depth_{R_q}(N_q) + \inf(M_{q}) -
  \big(\depth_{R_q}(N_q) - \depth_{R_q}(M_q   \ltensor_{R_q} N_q) + \cx(M_q) \big),
\end{equation*} 
which simplifies to 
\begin{equation} \label{main aim} \notag 
  \depth_{R_p}(N_p) \geq \depth_{R_q}(M_q \ltensor_{R_q} N_q) -
  \cx(M_q) + \inf(M_{q}).  
\end{equation} 
It is clear that $\cx_{R_{q}}(M_q) \leq \cx_{R}(M)=r$ and $\inf(M_q) +
\inf(N_q) = \inf(M_q \ltensor_{R_q} N_q)$. Moreover, $\inf(N_p) \geq
\inf(N_q)$.  Hence,
\begin{equation}  \label{final inequality} \tag{\ref{Cx}.7} %
 \begin{split} \hspace*{-2em}
 \depth_{R_p}(N_p) + \inf(N_p) 
 & \geq \depth_{R_p} (N_p) + \inf(N_q) \\
 & \geq \depth_{R_q}(M_q \ltensor_{R_q} N_q) - \cx(M_q) + \inf(M_q) +
 \inf(N_q) \\
 & = 
 \depth_{R_q}(M_q \ltensor_{R_q} N_q) + \inf(M_q \ltensor_{R_q} N_q) -
 \cx(M_q) \\
 & \geq \min\{n+r, \htt(q)\} - r = n, 
 \end{split}
\end{equation}
where the last equality follows from the fact that $\htt(q) > n + r$.
Thus the inequality (\ref{depth for N}) holds and hence $N$ satisfies $(S_{n})$.
\end{proof}

\section{Applications of Theorem \ref{Cx}}

\begin{cor} \label{cor1} 
Let $R$ be a local ring, $M$ and $N$ nonzero finitely generated
$R$-modules and let $n$ be a nonnegative integer.  Set
$r=\cx_{R}(M)$ and assume the following conditions hold:
\begin{enumerate}[(i)]
\item $\CI_{R}(M)<\infty$, 
\item $M\otimes_{R}N$ satisfies $(\displaystyle{S_{n+r}})$, and 
\item $\Tor_{>0}^{R}(M,N)=0$.
\end{enumerate}
Then $N$ satisfies $(S_{n})$.
\end{cor}

\begin{proof} By (iii),  $M\ltensor_{R}N \simeq
  M\otimes_{R}N$, so the conclusion follows from Theorem
  \ref{Cx}.
\end{proof}

Recall that, if $R$ is a complete intersection, then all finitely generated $R$-modules have finite complete intersection dimension; see \ref{qCI}. Therefore we have:

\begin{cor} \label{cor2} Let $R$ be a local complete intersection ring and let $M$ and $N$ be nonzero finitely generated $R$-modules. Set $r=\cx_{R}(M)$ and assume $n$ is a nonnegative integer. If $M\otimes_{R}N$ satisfies $(S_{n+r})$ and $\Tor^{R}_{>0}(M,N)=0$, then $N$ satisfies $(S_{n})$.
\end{cor}

The complexity of a module over a complete intersection cannot exceed the codimension of the ring; see \ref{cx}. This yields:

\begin{cor} \label{corintro}
Let $R$ be a local complete intersection ring of codimension $c$ and let $M$ and $N$ be nonzero finitely generated $R$-modules. Assume that $M\otimes_{R}N$ satisfies $(S_{n+c})$ for some nonnegative integer $n$. If $\Tor^{R}_{>0}(M,N)=0$, then $M$ and $N$ satisfy $(S_{n})$.
\end{cor}

The assertion that $M$ or $N$ is reflexive in the next result is implicitly contained in the proof of the Second Rigidity Theorem; see Theorem \ref{secondrigidity} and also Remark \ref{dis}.

\begin{cor} (Huneke--Wiegand \cite{HW1}) \label{cor3} 
Let $R$ be a local hypersurface ring and let $M$ and $N$ be nonzero finitely generated $R$-modules, either of which has constant rank. If $M\otimes_{R}N$ is reflexive, then (i) $M$ and $N$ are both torsion-free, and (ii) $M$ or $N$ is reflexive.
\end{cor}

\begin{proof} It follows from Theorem \ref{secondrigidity} that $\Tor^{R}_{>0}(M,N)=0$. This implies that either $M$ or $N$ has finite projective dimension \cite[1.9]{HW2}, that is, the complexity of $M$ or $N$ is zero. As the tensor product $M\otimes_{R}N$ satisfies $(S_{2})$, we conclude, setting $n=2$, that the claim of (ii) follows from Corollary \ref{cor2}. Furthermore, as the complexity of an $R$-module cannot exceed one, setting $n=1$ in Corollary \ref{cor2}, we obtain the required conclusion of (i).
\end{proof}

One can also obtain the conclusion of Corollary \ref{cor3}(i) from the proof of the Second Rigidity Theorem, see Theorem \ref{secondrigidity}: assume $R$, $M$ and $N$ are as in Corollary \ref{cor3}. Then $M\otimes_{R}N \cong (M/ \zdr R M)\otimes_{R}N$ and hence it follows from the Second Rigidity Theorem that $\Tor^{R}_{>0}(M/ \zdr R M,N)=0=\Tor^{R}_{>0}(M,N)$. Therefore, tensoring the short exact sequence $0 \to \zdr R M \to M \to M/ \zdr R M \to 0$ with $N$, we deduce that there is an injection $\zdr R M \otimes_{R}N \hookrightarrow M\otimes_{R}N$. This yields, as $N\neq 0$, that $\zdr R M=0$, that is, $M$ is torsion-free. Similarly one can see that $N$ is torsion-free, too.

We finish this section by pointing out, under the hypotheses of Corollary \ref{cor2}, the relationship between the depth and the complexity of $M$.

\begin{cor} \label{corfinal} Let $R$ be a local complete intersection ring of codimension $c$ and let $M$ and $N$ be nonzero finitely generated $R$-modules. Set $r=\cx_{R}(M)$ and assume $M\otimes_{R}N$ satisfies $(S_{c})$. Assume further that $\Tor^{R}_{>0}(M,N)=0$. 
\begin{enumerate}[(i)]
\item If $r<c$, then $N$ satisfies $(S_{c-r})$.
\item If $r=c$, then $M$ satisfies $(S_{c})$.
\end{enumerate}

\begin{proof} There is nothing to prove if $c=0$. Hence we may assume $c>0$. The claim of (i) immediately follows from Corollary \ref{cor2}. If $r=c$, then the vanishing of $\Tor^{R}_{\gg 0}(M,N)$ implies that $\pd(N)<\infty$ \cite[2.1]{Mi}. Therefore, since $M\otimes_{R}N$ is torsion-free, we have, as in  Remark $\ref{dis}$, that $\Supp(N)=\Spec(R)$ and hence $M$ satisfies $(S_{c})$.
\end{proof}

\end{cor}

\section*{Acknowledgments}
We are greateful to Roger Wiegand for explaining to us the observation in \ref{dis}. We thank  Tokuji Araya, Lars Winther Christensen, Hailong Dao, Sean Sather-Wagstaff, Ryo Takahashi, Siamak Yassemi and Yuji Yoshino for discussions related to this work during its preparation. Our thanks are also due to the anonymous referee(s) whose suggestions have improved the presentation of the paper.

\def\cfudot#1{\ifmmode\setbox7\hbox{$\accent"5E#1$}\else
  \setbox7\hbox{\accent"5E#1}\penalty 10000\relax\fi\raise 1\ht7
  \hbox{\raise.1ex\hbox to 1\wd7{\hss.\hss}}\penalty 10000 \hskip-1\wd7\penalty
  10000\box7}

\end{document}